\documentclass[12pt]{amsart}

\usepackage{amsthm, amssymb, booktabs, fullpage}

\newtheorem{theorem}{Theorem}
\newtheorem{lemma}{Lemma}[section]

\theoremstyle{remark}

\numberwithin{equation}{section}

\newcommand{\gam}{\gamma}
\newcommand{\eps}{\varepsilon}
\newcommand{\tet}{\theta}

\begin{document}

\title[The Waring--Goldbach problem for large powers]{On the Waring--Goldbach 
problem\\ for seventh and higher powers}
\author[A.V. Kumchev]{Angel V. Kumchev}
\address{Department of Mathematics\\ Towson University\\ Towson, MD 21252\\ USA}
\email{akumchev@towson.edu}
\author[T.D. Wooley]{Trevor D. Wooley}
\address{School of Mathematics\\ University of Bristol\\ University Walk\\ 
Bristol~BS8~1TW\\ UK}
\email{matdw@bristol.ac.uk}
\subjclass[2000]{11P32, 11L20, 11P05, 11P55.}

\begin{abstract} We apply recent progress on Vinogradov's mean value theorem to 
improve bounds for the function $H(k)$ in the Waring--Goldbach problem. We 
obtain new results for all exponents $k\ge 7$, and in particular establish that for large $k$ 
one has
$$H(k)\le (4k-2)\log k-(2\log 2-1)k-3.$$
\end{abstract}
\maketitle

\section{Introduction}In our recent work \cite{KW2016}, we reported on the consequences 
for the Waring-Goldbach problem of recent progress on Vinogradov's mean value theorem 
based on efficient congruencing (see, for example, \cite{Woo2012, Woo2015}). We now 
revisit our analysis in order to incorporate the latest developments stemming from work of 
Bourgain, Demeter and Guth \cite{BDG2015}. We first recall the definition of the function 
$H(k)$ associated with the Waring-Goldbach problem. Consider a natural number $k$ and 
prime number $p$, and define $\tet=\tet(k,p)$ to be the integer with $p^\tet|k$ but 
$p^{\tet+1}\nmid k$, and $\gam=\gam(k,p)$ by
$$\gam(k,p)=\begin{cases} \tet+2,&\text{when $p=2$ and $\tet>0$,}\\
\tet+1,&\text{otherwise.}
\end{cases}$$
We then put $K(k)=\prod_{(p-1)|k}p^\gam$, and denote by $H(k)$ the least integer $s$ 
such that every sufficiently large positive integer congruent to $s$ modulo $K(k)$ may be 
written as the sum of $s$ $k$-th powers of prime numbers.\par

Improving on the bound $H(k)\le k(4\log k+2\log\log k+O(1))$, as $k\to \infty$, due to 
Hua \cite{Hua59, Hua65}, we recently showed that $H(k)\le (4k-2)\log k+k-7$. The 
improved bound that we now present in this note saves roughly $(2\log 2)k$ further 
variables.

\begin{theorem}\label{th1}
When $k$ is large, one has $H(k)\le (4k-2)\log k-(2\log 2-1)k-3$.
\end{theorem}

For small values of $k$ one has the bounds
\begin{gather*}
H(1)\le 3, \quad H(2)\le 5,\quad H(3)\le 9, \quad H(4) \le 13, \quad H(5) \le 21, \quad H(6) 
\le 32, \quad H(7) \le 46,
\end{gather*}
as a consequence of work of Vinogradov \cite{IVin37}, Hua \cite{Hua38}, Kawada and the 
second author \cite{KaWo01}, the first author \cite{Kumc05}, and Zhao \cite{Zhao14}. For 
larger values of $k$, we recently established that
\begin{gather*}
H(8)\le 61, \quad H(9)\le 75,\quad H(10)\le 89, \quad H(11) \le 103, \quad H(12) \le 117,\\
H(13)\le 131,\quad H(14) \le 147,\quad H(15)\le 163,\quad H(16)\le 178,\\
H(17)\le 194,\quad H(18)\le 211,\quad H(19)\le 227,\quad H(20)\le 244.
\end{gather*}
We now obtain the following bounds for $H(k)$ when $7 \le k \le 20$.

\begin{theorem}\label{th2}
Let $7\le k\le 20$. Then $H(k)\le s(k)$, where $s(k)$ is defined by Table~\ref{tab1}.
\begin{table}[h]
\begin{center}
\begin{tabular}{ccccccccccccccc}
\toprule
$k$ & $7$ & $8$  & $9$  & $10$ & $11$  & $12$  & $13$  & $14$  & $15$  & $16$  & $17$  & 
$18$  & $19$  & $20$ \\
$s(k)$ & $45$ & $57$ & $69$ & $81$ & $93$ & $107$ & $121$ & $134$ & $149$ & $163$ & 
$177$ & $193$ & $207$ & $223$\\
\bottomrule
\end{tabular}\\[6pt]
\end{center}
\caption{Upper bounds for $H(k)$ when $7\le k\le 20$}\label{tab1}
\end{table}
\end{theorem}

Our proof of Theorems \ref{th1} and \ref{th2} proceeds by directly incorporating the 
refinements available via \cite{BDG2015} into our previous methods from \cite{KW2016}. 
We record improved estimates for Weyl sums in \S2, both pointwise bounds and mean 
value estimates. Then, in \S3, we indicate how to refine our previous bounds for $H(k)$ 
using these bounds, thereby establishing Theorems \ref{th1} and \ref{th2}.\par

Throughout this paper, the letter $\eps$ denotes a sufficiently small positive number. 
Whenever $\eps$ occurs in a statement, we assert that the statement holds for each 
positive $\eps$, and any implied constant in such a statement is allowed to depend on 
$\eps$. The letter $p$, with or without subscripts, is reserved for prime numbers. We also 
write $e(x)$ for $\exp(2\pi \mathrm{i}x)$, and $(a, b)$ for the greatest common divisor 
of $a$ and $b$. Finally, for real numbers $\theta$, we denote by $\lfloor \theta\rfloor$ the 
largest integer not exceeding $\theta$, and by $\lceil \theta\rceil$ the least integer no 
smaller than $\theta$.\par

\section{Auxiliary estimates for exponential sums} We refine the work of 
\cite[\S\S2 and 3]{KW2016} by incorporating recent progress on Vinogradov's mean value 
theorem due to Bourgain, Demeter and Guth \cite{BDG2015}. Recall the classical Weyl sum
\[
f_k(\alpha; X)=\sum_{X < x \le 2X}e\left( \alpha x^k \right),
\]
in which we suppose that $k \ge 2$ is an integer and $\alpha$ is real. When $k \ge 3$ 
is an integer, we define $\sigma_k$ by means of the relation
\begin{equation}\label{2.1}
\sigma_k^{-1} = \min\big\{ 2^{k-1}, k(k-1)\big\}. 
\end{equation}
Also, for $k\ge 3$, we define the multiplicative function $w_k(q)$ by taking
\[
w_k(p^{uk+v})=\begin{cases} kp^{-u-1/2},&\text{when $u\ge 0$ and $v=1$,}\\
p^{-u-1},&\text{when $u\ge 0$ and $2\le v\le k$.}
\end{cases}
\]

\begin{lemma}\label{lem2.1.1}
Suppose that $k\ge 3$. Then either one has $f_k(\alpha; X)\ll X^{1-\sigma_k+\eps}$, or 
there exist integers $a$ and $q$ such that $1\le q\le X^{k\sigma_k}$, $(a, q)=1$ and 
$|q\alpha-a|\le X^{-k+k\sigma_k}$, in which case 
\[
f_k(\alpha;X)\ll \frac {w_k(q)X}{1+X^k|\alpha -a/q|}+X^{1/2+\eps}.
\]
\end{lemma}

\begin{proof} One may apply the argument of the proof of \cite[Lemma 2.1]{KW2016}, 
noting only that the refinement of \cite[Theorem 11.1]{Woo2015} that follows by 
employing the bounds recorded in \cite[Theorem 1.1]{BDG2015} permits the use of the 
exponent $\sigma_k$ with the revised definition (\ref{2.1}) presented above.
\end{proof}

We also require upper bounds for the corresponding Weyl sum over prime numbers,
\[
g_k(\alpha;X)=\sum_{X <p\le 2X}e\left( \alpha p^k \right),
\] 
and these we summarise in the next lemma.

\begin{lemma}\label{lem2.1.2}
Suppose that $k \ge 4$ and $X^{2\sigma_k/3} \le P \le X^{9/20}$. Then either one has 
the bound $g_k(\alpha; X) \ll X^{1-\sigma_k/3+\eps}$, or else there exist integers $a$ and 
$q$ such that $1\le q\le P$, $(a,q)=1$ and $|q\alpha-a|\le PX^{-k}$, in which case 
\begin{equation}\label{2.1.6}
g_k(\alpha;X)\ll \frac{X^{1+\eps}}{(q+X^k|q\alpha-a|)^{1/2}}.
\end{equation}
\end{lemma}

\begin{proof} One may follow the argument of the proof of \cite[Lemma 2.2]{KW2016}, 
noting that the refinement to the exponent $\sigma_k$ made available via (\ref{2.1}) as 
exhibited in Lemma \ref{lem2.1.1}.
\end{proof}

In order to describe our critical mean-value estimate, we introduce a set of admissible 
exponents for $k$th powers as follows. Let $t=t_k$ and $u=u_k$ be positive integers to 
be fixed in due course. Put $\theta =1-1/k$, and define 
\begin{equation}\label{lam.1}
\lambda_i=(\theta +\sigma_{k-1}/k)^{i-1}\quad (1\le i\le u+1).
\end{equation}
Then define $\lambda_{u+2}, \dots, \lambda_{u+t}$ by putting
\begin{align}
\lambda_{u+2}&=\frac{k^2-\theta^{t-3}}{k^2+k-k\theta^{t-3}}\lambda_{u+1},
\label{lam.2}\\
\lambda_{u+j}&=\frac {k^2-k-1}{k^2+k-k\theta^{t-3}}\theta^{j-3}\lambda_{u+1}\quad 
(3\le j\le t),\label{lam.3}
\end{align}
and then write
\begin{equation}\label{lam.4}
\Lambda=\lambda_1+\ldots +\lambda_{t+u}.
\end{equation}

\begin{lemma}\label{lem2.3.2}
Let $k$, $t$ and $u$ be positive integers with $k\ge 3$ and 
$t\ge \left\lfloor \frac 12(k+3)\right\rfloor$, and let $w$ be a non-negative integer. Define 
the exponents $\lambda_j$ and $\Lambda$ by means of \eqref{lam.1}-\eqref{lam.4}, and 
put $\eta=\max\{0,k-\Lambda -2w\sigma_k\}$. Then when $N$ is sufficiently large, one 
has
\[
\int_0^1 |g_k(\alpha; N)|^{2w}\prod_{j=1}^{t+u}\left| 
g_k (\alpha;N^{\lambda_j})\right|^2 \, {\rm d}\alpha \ll 
N^{2\Lambda+2w-k+\eta+\eps}.
\]
\end{lemma}

\begin{proof} This is \cite[Lemma 3.3]{KW2016}, modified to reflect the improved Weyl 
exponent (\ref{2.1}) as exhibited in Lemmata \ref{lem2.1.1} and \ref{lem2.1.2}.
\end{proof}

\section{The upper bound for $H(k)$}
An upper bound for $H(k)$ follows by combining the mean value estimate supplied by 
Lemma \ref{lem2.3.2} with the Weyl-type estimate stemming from Lemma 
\ref{lem2.1.2}.

\begin{lemma}\label{lemw.1}
Let $k$, $t$ and $u$ be positive integers with $k\ge 3$ and 
$t\ge \left\lfloor\frac 12(k+3)\right\rfloor$. Define the exponent $\Lambda$ by means of 
\eqref{lam.4}, and put $v=\lfloor (k-\Lambda)/(2\sigma_k)\rfloor$ and $\eta^*=
k-\Lambda -2v\sigma_k$. Finally, define
\[
h=\begin{cases}1,&\text{when $0\le \eta^*<\tfrac{1}{2}\sigma_k$,}\\
2,&\text{when $\tfrac{1}{2}\sigma_k\le \eta^*<\sigma_k$,}\\
3,&\text{when $\sigma_k\le \eta^*<2\sigma_k$.}\end{cases}
\]
Suppose in addition that $2(t+u+v)+h\ge 3k+1$ and, when $h\in \{1,2\}$, that either 
$v\ge 3$ or $\eta^*<h\sigma_k/3$. Then
\[
H(k)\le 2(t+u+v)+h.
\]
\end{lemma}

\begin{proof} This is \cite[Lemma 4.1]{KW2016}, modified to reflect the improved Weyl 
exponent (\ref{2.1}) as exhibited in Lemma \ref{lem2.1.1}-\ref{lem2.3.2}.
\end{proof}

We establish Theorems \ref{th1} and \ref{th2} by applying Lemma \ref{lemw.1}. 
Recall (\ref{lam.1})-(\ref{lam.4}), and write $\sigma=\sigma_{k-1}$ and 
$\phi=\theta+\sigma/k$. Then, just as in the discussion of \cite[\S5]{KW2016}, one has
\[
k-\Lambda=-\frac{k\sigma}{1-\sigma}+\left( \frac{k^2(k+1)\sigma +\theta^{t-3}
((k^3-3k^2+k+2)-\sigma(k^3-2k^2+k+2))}{(k^2+k-k\theta^{t-3})(1-\sigma)}\right) 
\phi^u.
\]

\begin{proof}[The proof of Theorem \ref{th2}] Let $k$ be an integer with $7\le k\le 20$, 
and define $t =t_k$, $u=u_k$, $v=v_k$ and $h=h_k$ by means of Table \ref{tab2}. Then 
by application of a simple computer program, one confirms the validity of the hypotheses 
of Lemma \ref{lemw.1}. Thus $H(k)\le 2(t+u+v)+h$. Indeed, with $h_k^*$ defined as in 
Table \ref{tab3}, one finds that for each $k$ one has $2\eta^*/\sigma_k<h_k^*$. We 
note in this context that the entries in this table have been rounded up in the final decimal 
place presented. This completes our proof of Theorem \ref{th2}.

\begin{table}[ht]
\begin{center}
\begin{tabular}{ccccccccccccccc}
\toprule
$k$ & $7$ & $8$ & $9$ & $10$ & $11$ & $12$ & $13$ & $14$ & $15$ & $16$ & $17$ & $18$ & 
$19$ & $20$ \\
\midrule
$t_k$ & $7$ & $12$ & $17$ & $10$ & $13$ & $10$ & $24$ & $19$ & $30$ & $17$ & $25$ & 
$18$ & $29$ & $37$ \\
$u_k$ & $13$ & $12$ & $14$ & $25$ & $29$ & $37$ & $28$ & $42$ & $41$ & $60$ & $56$ & 
$74$ & $66$ & $63$ \\
$v_k$ & $2$ & $4$ & $3$ & $5$ & $4$ & $6$ & $8$ & $5$ & $3$ & $4$ & $7$ & $4$ & $8$ & 
$11$ \\
$h_k$ & $1$ & $1$ & $1$ & $1$ & $1$ & $1$ & $1$ & $2$ & $1$ & $1$ & $1$ & $1$ & $1$ & 
$1$ \\
\bottomrule
\end{tabular}\\[12pt]
\end{center}
\caption{The values of $t_k$, $u_k$, $v_k$ and $h_k$ for $7 \le k \le 20$}\label{tab2}
\end{table}

\begin{table}[ht]
\begin{center}
\begin{tabular}{cccccccccccccc}
\toprule
$k$ & $7$ & $8$ & $9$ & $10$ & $11$ & $12$ & $13$  \\
\midrule
$h_k^*$ & $0.44643$ & $0.22927$ & $0.02678$ & $0.00739$ & $0.97975$ & $0.00042$ & 
$0.08628$ \\
\bottomrule \toprule
$k$ & $14$ & $15$ & $16$ & $17$ & $18$ & $19$ & $20$\\
\midrule
$h_k^*$ & $1.94435$ & $0.03925$ & $0.01091$ & $0.39085$ & $0.00541$ & $0.52855$ & 
$0.00043$\\
\bottomrule
\end{tabular}\\[12pt]
\end{center}
\caption{The values of $h_k^*$ for $7 \le k \le 20$}\label{tab3}
\end{table}

\end{proof}

As we remarked in \cite[\S5]{KW2016}, the non-monotonicity in the values of $t_k$, 
$u_k$ and $v_k$ recorded in Table \ref{tab2} is a consequence of the fact that $\theta$ 
and $\phi$ are close in size, and thus the optimisation is sensitive only to the sum 
$t_k+u_k$ rather than the individual values of $t_k$ and $u_k$.

\begin{proof}[The proof of Theorem \ref{th1}]
We adapt the proof of \cite[Theorem 1]{KW2016}, supposing throughout that $k$ is 
sufficiently large. Put $t=t_k$ and $u=u_k$, where
\[
t_k=\left\lceil \tfrac{1}{2}k\log k\right\rceil \quad \text{and}\quad u_k=
\left\lceil k(2\log k-\log 2)\right\rceil -t.
\]
It is convenient to define $\gam=\lceil k(2\log k-\log 2)\rceil -k(2\log k-\log 2)$. Also, we 
write
\[
\tau=\frac{1}{k(k-1)}\quad \text{and}\quad \sigma=\frac{1}{(k-1)(k-2)},
\]
so that $\sigma_k=\tau$ and $\sigma_{k-1}=\sigma$. Our earlier formula for 
$k-\Lambda$ now takes the shape
\[
k-\Lambda=-\frac{k\sigma}{1-\sigma}+\left( \frac{k^2(k+1)(k-1)^3\sigma +\theta^t
k^3(k^3-3k^2+O(k))}{(k-1)^3(k^2+k-k\theta^{t-3})(1-\sigma)}\right) 
\phi^u.
\]

\par As in the corresponding proof of \cite[Theorem 1]{KW2016}, one finds that
\[
\theta^t=e^{-t/k}\left( 1-\frac{\log k}{4k}+O(k^{-3/2})\right)\asymp k^{-1/2}.
\]
Also, since
\[
\log \phi=\log \left( 1-\frac{1-\sigma}{k}\right)=-\frac{1}{k}-\frac{1}{2k^2}+
O\left(\frac{1}{k^3}\right),
\]
one discerns that
\[
\phi^u=e^{-u/k}\left( 1-\frac{3\log k-2\log 2}{4k}+O(k^{-3/2})\right) \ll k^{-3/2}.
\]
Consequently,
\[
k-\Lambda=-\frac{k\sigma}{1-\sigma}+\left(k-1+O(k^{-1/2})\right)\theta^t\phi^u
+O(k^{-5/2}),
\]
where
\begin{align*}
\theta^t\phi^u&=e^{-(t+u)/k}\left( 1-\frac{2\log k-\log 2}{2k}+O(k^{-3/2})\right)\\
&=e^{-\gamma/k}\left( \frac{2}{k^2}-\frac{2\log k-\log 2}{k^3}+O(k^{-7/2})\right).
\end{align*}

\par Since
\[
\frac{\sigma}{\tau}=\frac{k(k-1)}{(k-1)(k-2)}=1+\frac{2}{k}+O\left( 
\frac{1}{k^2}\right),
\]
we find that
\begin{align*}
\frac{k-\Lambda}{2\tau}&=-\tfrac{1}{2}(k+2)+
\frac{e^{-\gam/k}k(k-1)}{k^3}(k-1)(k-\log k+\tfrac{1}{2}\log 2)+O(k^{-1/2})\\
&=-\tfrac{1}{2}(k+2)+(k-\log k+\tfrac{1}{2}\log 2-2)\left( 1-\gamma/k\right)
+O(k^{-1/2})\\
&=\tfrac{1}{2}k-\log k-3+\tfrac{1}{2}\log 2-\gamma+O(k^{-1/2}).
\end{align*}
Put $v=\left\lfloor(k-\Lambda)/(2\tau)\right\rfloor$, set $\eta^*=k-\Lambda-2v\tau$, and 
define $h$ as in the statement of Lemma \ref{lemw.1}. Then one has 
$0\le \eta^*<2\tau$, and in all circumstances one may confirm that
\[
2v+h=\frac{k-\Lambda-\eta^*}{\tau}+h\le \frac{k-\Lambda}{\tau}+2
\le k-2\log k-4+\log 2-2\gamma+O(k^{-1/2}).
\]
Since
\[
2(t+u+v)+h\le 2(2k\log k-k\log 2+\gamma)+k-2\log k-4+\log 2-2\gamma+O(k^{-1/2}),
\]
we therefore conclude from Lemma \ref{lemw.1} that
\[
H(k)\le (4k-2)\log k-(2\log 2-1)k-4+\log 2+O(k^{-1/2}).
\]
We have assumed $k$ to be sufficiently large, and thus we have established the bound
\[
H(k)\le (4k-2)\log k-(2\log 2-1)k-3.
\]
This completes the proof of Theorem \ref{th1}.
\end{proof}

\providecommand{\bysame}{\leavevmode\hbox to3em{\hrulefill}\thinspace}

\end{document}